\documentclass[reqno]{amsart}
\usepackage{graphicx}
\usepackage{color}
\usepackage{ThinKhTorsion}

\let\stdhline\hline
\let\TSp\thinspace
\def\DSp{\thinspace\thinspace}
\def\lbr{\raise 1pt\hbox{[}}
\def\rbr{\raise 1pt\hbox{]}}

\def\gobble#1{}
\def\hline{\multispan\numcolumns\hrulefill\cr}
\def\torframe#1{\vtop{\vbox{\hrule\hbox{\vrule\strut #1\vrule}}\hrule}}
\def\dblhline{\hline height 0.16667em\gobble&\emptyline\cr\hline}

\newbox\tablebox

\let\lra\longrightarrow
\def\vlra{\DOTSB\protect\relbar\protect\joinrel
  \protect\relbar\protect\joinrel\rightarrow}

\begin{document}
\title[Torsion in Khovanov homology of homologically thin knots]
{Torsion in Khovanov homology\\ of homologically thin knots}
\author[A.~Shumakovitch]{Alexander Shumakovitch}
\address{Department of Mathematics, The George Washington University,
Phillips Hall, 801\ 22nd St. NW, Suite \#739, Washington, DC 20052, U.S.A.}
\email{Shurik@gwu.edu}
\thanks{The author is partially supported by a Simons Collaboration Grant for
Mathematicians~\#279867}
\subjclass[2010]{57M25, 57M27}
\keywords{Khovanov homology, torsion in Khovanov homology, homologically thin
links, Bockstein spectral sequence}
\begin{abstract}
We prove that every $\Z_2$H-thin link has no $2^k$-torsion for \hbox{$k>1$} in
its Khovanov homology. Together with previous results by Eun Soo
Lee~\cite{Lee-H_slim,Lee-patterns} and the author~\cite{me-torsion}, this implies
that integer Khovanov homology of non-split alternating links is completely
determined by the Jones polynomial and signature. Our proof is based on
establishing an algebraic relation between Bockstein and Turner differentials
on Khovanov homology over $\Z_2$. We conjecture that a similar relation
exists between the corresponding spectral sequences.
\end{abstract}
\maketitle

\setcounter{footnote}{1}

\vskip -0.65cm\null
\section{Introduction}
Let $L$ be an oriented link in the Euclidean space $\R^3$ represented by a
planar diagram~$D$. In a seminal paper~\cite{Kh-Jones}, Mikhail Khovanov
assigned to $D$ a family of abelian groups $\CalH^{i,j}(L)$, whose isomorphism
classes depend on the isotopy class of $L$ only. These groups are defined as
homology groups of an appropriate (graded) chain complex $\CalC(D)$ with
integer coefficients.
The main property of the Khovanov homology is that it {\em categorifies} the
Jones polynomial. More specifically, let $J_L(q)$ be a version of the Jones
polynomial of $L$ that satisfies the following skein relation and
normalization:
\begin{equation}\label{eq:K-jones-skein}
-q^{-2}J_{\includegraphics[scale=0.45]{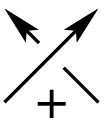}}(q)
+q^2J_{\includegraphics[scale=0.45]{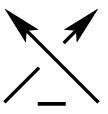}}(q)
=(q-1/q)J_{\includegraphics[scale=0.45]{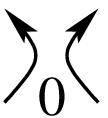}}(q);
\qquad
J_{\includegraphics[scale=0.45]{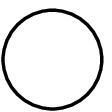}}(q)=q+1/q.
\end{equation}
Then $J_L(q)$ equals the (graded) Euler characteristic of the Khovanov chain
complex:
\begin{equation}\label{eq:khpol-jones}
J_L(q)=\chi_q(\CalC(D))=\sum_{i,j}(-1)^iq^jh^{i,j}(L),
\end{equation}
where $h^{i,j}(L)=\rk(\CalH^{i,j}(L))$, the Betti numbers of $\CalH(L)$.
We explain the main ingredients of Khovanov's construction
in Section~\ref{sec:ingreds}. The reader is referred
to~\cite{BN-first,Kh-Jones} for a detailed treatment.

In this paper, we consider a more general setup by allowing
$\CalC(D)$ to consist of free $R$-modules, where $R$ is a commutative ring
with unity. This results in a Khovanov homology theory with coefficients in
$R$. We are mainly interested in the cases when $R=\Z$, $\Q$, or $\Z_2$.

\begin{defins}[\cite{me-torsion}, cf. Khovanov~\cite{Kh-patterns}]
A link $L$ is said to be {\em homologically thin} over a ring $R$ or simply
{\em $R$H-thin} if its Khovanov homology with coefficients in $R$
is supported on two adjacent diagonals $2i-j=const$. 
A link $L$ is said to be {\em homologically slim} or simply {\em H-slim} if it
is $\Z$H-thin and all its homology groups that are supported on the {\em upper}
diagonal have no torsion.
A link $L$ that is not $R$H-thin is said to be {\em $R$H-thick}.
\end{defins}

\begin{prop}\label{prop:all-thins}
If a link is $\Z$H-thin, then it is $\Q$H-thin as well by definition.
Consequently, a $\Q$H-thick link is necessarily $\Z$H-thick. If a link is
H-slim, then it is $\Z_m$H-thin for every $m>1$ by the Universal Coefficient
Theorem.
\end{prop}

\begin{example}
Most of the $\Z$H-thin knots are H-slim. The first prime $\Z$H-thin knot that
is not H-slim is the mirror image of $16^n_{197566}$\footnote{This denotes the
non-alternating knot number 197566 with 16 crossings from the Knotscape knot
table~\cite{Knotscape}.}.
It is also $\Z_2$H-thick, see Figure~\ref{fig:diff-thickness}. In these
tables, columns and rows are marked with $i$- and $j$-grading of the Khovanov
homology, respectively. Only entries representing non-trivial groups are
shown. An entry of the form $a,b_2,c_4$ means that the corresponding group is
$\Z^a\oplus\Z_2^b\oplus\Z_4^c$.
\end{example}

H-slim knots possess several important properties that were observed
in~\cite{BN-first,Garoufalidis,me-torsion} and proved
in~\cite{Lee-H_slim,Lee-patterns,me-torsion}. We list them below.

\begin{figure}
\centerline{\resizebox{\hsize}{!}{\input knot-16n_197566-table}}
\par\medskip
Knot $16^n_{197566}$ is $\Q$H-thin, but $\Z$H-thick and $\Z_2$H-thick.
\par\bigskip\bigskip\bigskip
\centerline{\resizebox{\hsize}{!}{\input knot-16n_-197566-table}}
\par\medskip
Mirror image of the knot $16^n_{197566}$ is $\Q$H-thin and
$\Z$H-thin, but $\Z_2$H-thick.
\par\bigskip
\caption{Integral Khovanov homology of the knot $16^n_{197566}$ and
its mirror image.}\label{fig:diff-thickness}
\vskip 1in 
\end{figure}

\begin{thm}[Lee~\cite{Lee-H_slim,Lee-patterns}]\label{thm:lee-thin}
Every oriented non-split alternating link $L$ is H-slim and the Khovanov
homology of $L$ is supported on the diagonals $2i-j=\Gs(L)\pm1$, where
$\Gs(L)$ is the signature of $L$.
\end{thm}

Let $\tJ_L(q)=J_L(q)/(q+1/q)$ be a renormalization of $J_L(q)$ that 
equals $1$ on the unknot instead of $q+1/q$.

\begin{thm}[Lee~\cite{Lee-patterns}]\label{thm:lee-pattern}
If $L$ is an H-slim link, then its rational Khovanov homology $\CalH_\Q(L)$ is
completely determined, up to a grading shift, by the Jones polynomial $J_L(q)$
of $L$. In particular, the total rank of $\CalH(L)$ is given by
$\rk\CalH_\Q(L)=|\tJ_L(\sqrt{-1})|+2^{c-1}$, where $c$ is the number of
components of $L$ {\rm (cf.~\cite{Garoufalidis}).}
\end{thm}

\begin{cor}
Every oriented non-split alternating link $L$ has its rational Khovanov
homology $\CalH_\Q(L)$ completely determined by $J_L(q)$ and $\Gs(L)$.
\end{cor}

\begin{thm}[\cite{me-torsion}]\label{thm:me-pattern}
If $L$ is an H-slim link, then its integer Khovanov homology $\CalH(L)$ has no
torsion elements of odd order and its Khovanov homology $\CalH_{\Z_2}(L)$ with
coefficients in $\Z_2$ is completely determined, up to a grading shift, by the
Jones polynomial $J_L(q)$ of $L$. In particular, the total dimension of
$\CalH_{\Z_2}(L)$ over $\Z_2$ is given by
$\dim_{\Z_2}(\CalH_{\Z_2}(L))=2|\tJ_L(\sqrt{-1})|$.
\end{thm}

\begin{cor}
Every oriented non-split alternating link $L$ has its integer Khovanov
homology $\CalH(L)$ all but determined by $J_L(q)$ and $\Gs(L)$, except that
one cannot distinguish between $\Z_{2^k}$ factors in the canonical
decomposition of $\CalH(L)$ for different values of $k$.
\end{cor}

\begin{rem}
\cite{Lee-patterns} and~\cite{me-torsion} contain much more information about
the structure of $\CalH_\Q(L)$ and $\CalH(L)$ of H-slim links than 
Theorems~\ref{thm:lee-pattern} and~\ref{thm:me-pattern}, respectively. But we
have no use of more general statements in this paper.
\end{rem}

It was conjectured in~\cite{me-torsion} that $\CalH(L)$ of H-slim links can
only contain $2$-torsion. Our main result is to prove this conjecture.

\begin{thm}\label{thm:main-result}
Let $L$ be a $\Z_2$H-thin link. Then $\CalH(L)$ contains no torsion elements
of order $2^k$ for $k>1$. 
\end{thm}

\begin{cor}
If $L$ is an H-slim link, then every non-trivial torsion element of $\CalH(L)$
has order $2$. Consequently, every oriented non-split alternating link $L$ has
its integer Khovanov homology completely determined by $J_L(q)$ and $\Gs(L)$.
\end{cor}

We end this section with a long-standing conjecture from~\cite{me-torsion} 
that provides another motivation for studying $2$-torsion in the Khovanov
homology. Partial results in this directions were obtained 
in~\cite{Asaeda-Przytycki,Jozef-Radmila+Milena-torsion,Jozef-Radmila-torsion}.

\begin{myconj}\label{conj:tors-existence}
Khovanov homology of every non-split link except the trivial knot, the Hopf
link, and their connected sums contains torsion elements of order $2$.
\end{myconj}

This paper is organized as follows. In Section~\ref{sec:ingreds} we
recall main definitions and facts about the Khovanov homology and auxiliary
constructions that are going to be used in the paper.
Section~\ref{sec:Bockstein} contains brief overview of the Bockstein spectral
sequence construction following~\cite{McCleary-Bockstein}
as well as a proof of Theorem~\ref{thm:main-result} modulo technical
Lemma~\ref{lem:diff-commucator} whose proof is postponed until
Section~\ref{sec:lemma-proof}.

\subsection*{Acknowledgements}
The author is grateful to Alain Jeanneret for bringing the Bockstein spectral
sequence to his attention.

\section{Main ingredients and definitions}\label{sec:ingreds}
In this section we give a brief outline of the Khovanov homology theory
following~\cite{Kh-Jones}. We also recall required bits and pieces
from~\cite{Turner-diff} and~\cite{me-torsion}.

\subsection{Algebraic preliminaries}
Let $R$ be a commutative ring with unity. In this paper, we are only
interested in the cases $R=\Z$, $\Q$, or $\Z_2$. If $M$ is a graded
$R$-module, we denote its {\em homogeneous component of degree $j$} by $M_j$.
For an integer $k$, the {\em shifted module} $M\{k\}$ is defined as having
homogeneous components $M\{k\}_j=M_{j-k}$. In the case when $M$ is free and
finite dimensional, we
define its {\em graded dimension} as the Laurent polynomial
$\dim_q(M)=\sum_{j\in\Z}q^j\dim(M_j)$ in variable~$q$. Finally, 
if $(\CalC,d)=\Bigl(\cdots\lra\CalC^{i-1}\stackrel{d^{i{-}1}}{\lra}\CalC^i
\stackrel{d^i}{\lra}\CalC^{i+1}\lra\cdots\Bigr)$ is a (co)chain complex of
graded free $R$-modules such that all differentials $d^i$ are graded of
degree~$0$ with respect to the internal grading, we define its {\sl graded
Euler characteristic} as $\chi_q(\CalC)=\sum_{i\in\Z}(-1)^i\dim_q(\CalC^i)$.

\begin{rem}
One can think of a chain complex of graded $R$-modules as a {\em bigraded
$R$-module} where the homogeneous components are indexed by pairs of numbers
$(i,j)\in\Z^2$. Under this point of view, the differentials are graded of
{\em bidegree} $(1,0)$.
\end{rem}

Let $A=R[X]/X^2$, the ring of truncated polynomials. As an $R$-module, $A$ is
freely generated by $1$ and $X$. We put grading on $A$ by specifying that
$\deg(1)=1$ and $\deg(X)=-1$. In other words, $A\simeq R\{1\}\oplus R\{-1\}$
and $\dim_q(A)=q+q^{-1}$. At the same time, $A$ is a (graded) commutative
algebra with the unit $1$ and multiplication $m:A\otimes A\to A$ given by
\begin{equation}\label{eq:A-mult}
m(1\otimes 1)=1,\qquad m(1\otimes X)=m(X\otimes1)=X,\qquad m(X\otimes X)=0.
\end{equation}
Algebra $A$ can also be equipped with comultiplication $\Delta:A\to A\otimes
A$ defined as
\begin{equation}\label{eq:A-comult}
\GD(1)=1\otimes X+X\otimes 1,\qquad \GD(X)=X\otimes X.
\end{equation}
It follows directly from the definition that $m$ and $\GD$ are graded maps
with
\begin{equation}
\label{eq:cobord-grading}
\deg(m)=\deg(\GD)=-1.
\end{equation}

\begin{rem}
Together with a counit map $\Ge:A\to R$ given by $\Ge(1)=0$ and $\Ge(X)=1$,
$A$ has a structure of a commutative Frobenius algebra over $R$,
see~\cite{Kh-Frobenius}.
\end{rem}

\subsection{Khovanov chain complex}\label{sec:Kh-complex}
Let $L$ be an oriented link and $D$ its planar diagram. We assign a number
$\pm1$, called {\em sign}, to every crossing of $D$ according to the rule
depicted in Figure~\ref{fig:crossing-signs}. The sum of these signs over all
the crossings of $D$ is called the {\em writhe number} of $D$ and is denoted
by $w(D)$.

\begin{figure}
\captionindent 0.35\captionindent
\begin{minipage}{1.8in}
\centerline{\input{Xing_signs.pspdftex}}
\caption{Positive and negative crossings}
\label{fig:crossing-signs}
\end{minipage}
\hfill
\begin{minipage}{3.1in}
\centerline{\input{markers.pspdftex}}
\caption{Positive and negative markers and the corresponding resolutions of a
diagram.}
\label{fig:markers}
\end{minipage}
\end{figure}

Every crossing of $D$ can be {\em resolved} in two different ways according to
a choice of a {\em marker}, which can be either {\em positive} or {\em
negative}, at this crossing (see Figure~\ref{fig:markers}). A collection of
markers chosen at every crossing of a diagram $D$ is called a {\em (Kauffman)
state} of $D$. For a diagram with $n$ crossings, there are, obviously, $2^n$
different states. Denote by $\Gs(s)$ the difference between the numbers of
positive and negative markers in a given state $s$. Define
\begin{equation}\label{eq:state-ij}
i(s)=\frac{w(D)-\Gs(s)}2,\qquad j(s)=\frac{3w(D)-\Gs(s)}2.
\end{equation}
Since both $w(D)$ and $\Gs(s)$ are congruent to $n$ modulo $2$, $i(s)$ and
$j(s)$ are always integer. For a given state $s$, the result of
the resolution of $D$ at each crossing according to $s$ is a family $D_s$ of
disjointly embedded circles. Denote the number of these circles by $|D_s|$.

For each state $s$ of $D$, let $\CalA(s)=A^{\otimes|D_s|}\{j(s)\}$. One
should understand this construction as assigning a copy of algebra $A$ to each
circle from $D_s$, taking the tensor product of all of these copies, and
shifting the grading of the result by $j(s)$. By construction, $\CalA(s)$ is
a graded free $R$-module of graded dimension 
$\dim_q(\CalA(s))=q^{j(s)}(q+q^{-1})^{|D_s|}$. Let
$\CalC^i(D)=\bigoplus_{i(s)=i}\CalA(s)$ for each $i\in\Z$. It is easy to check
(see~\cite{BN-first,Kh-Jones}) that $\chi_q(\CalC(D))=J_L(q)$, that is, the
graded Euler characteristic of $\CalC(D)$ equals the Jones polynomial of the
link $L$.

\begin{figure}
\centerline{\vbox{\halign{#\hfill\cr\hskip-0.5em\input{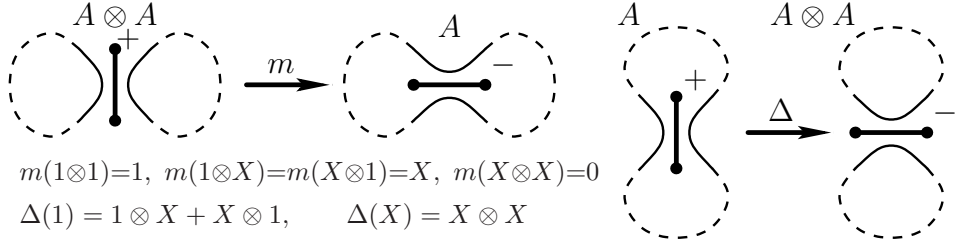}\cr
\noalign{\vskip -1.2\baselineskip}
\hbox{\vbox to0pt{\vss\halign{#\hfill\cr
$m(1{\otimes}1){=}1$,\enspace $m(1{\otimes}X){=}m(X{\otimes}1){=}X$,\enspace
$m(X{\otimes}X){=}0$\cr
\noalign{\medskip}
$\GD(1)=1\otimes X+X\otimes 1,\qquad\GD(X)=X\otimes X$\cr
\noalign{\smallskip}}}}\cr
}}}
\caption{Diagram resolutions corresponding to adjacent states and 
maps between the algebras assigned to the circles}\label{fig:res-change}
\end{figure}

In order to make $\CalC(D)$ into a graded complex, we need to define a
(graded) differential $d^i:\CalC^i(D)\to\CalC^{i+1}(D)$ of degree $0$.
Let $s_+$ and $s_-$ be two states of $D$ that differ at a single
crossing, where $s_+$ has a positive marker while $s_-$ has a negative one.
We call two such states {\em adjacent}. In this case, $\Gs(s_-)=\Gs(s_+)-2$
and, consequently, $i(s_-)=i(s_+)+1$ and $j(s_-)=j(s_+)+1$. Consider now the
resolutions of $D$ corresponding to $s_+$ and $s_-$. One can
readily see that $D_{s_-}$ is obtained from $D_{s_+}$ by either merging two
circles into one or splitting one circle into two (see
Figure~\ref{fig:res-change}). We define $d_{s_+:s_-}:\CalA(s_+)\to\CalA(s_-)$
as either $m$ or $\GD$ depending on whether the circles merge or split.
All the circles that do not pass through the crossing at which $s_+$ and $s_-$
differ, remain unchanged and $d_{s_+:s_-}$ acts as the identity map on the
corresponding copies of $\CalA$. Since $j(s_+)-j(s_-)=1$,
\eqref{eq:cobord-grading} ensures that $\deg(d_{s_+:s_-})=0$.

Finally, let $d^i=\sum_{(s_+,s_-)}\Ge(s_+,s_-)d_{s_+:s_-}$, where
$i(s_+)=i$ and $(s_+,s_-)$ runs over all adjacent pairs of states.
$\Ge(s_+,s_-)=\pm1$ is a sign that can be defined explicitly~\cite{Kh-Jones},
but for our purposes it is enough to know that it only depends on the crossing
at which $s_+$ and $s_-$ differ. It is straightforward to verify that
$d^{i+1}\circ d^i=0$ and, hence, $d:\CalC(D)\to\CalC(D)$ is indeed
a differential.

\begin{defin}[Khovanov,~\cite{Kh-Jones}]\label{def:Khovanov}
The resulting (co)chain complex
$\CalC(D)=\Bigl(\cdots\lra\CalC^{i-1}(D)
\stackrel{d^{i{-}1}}{\lra}\CalC^i(D)
\stackrel{d^i}{\lra}\CalC^{i+1}(D)\lra\cdots\Bigr)$
is called the {\em Khovanov chain complex} of the diagram $D$.
The homology of $\CalC(D)$ with respect to $d$ is called the {\em Khovanov
homology} of $L$ and is denoted by $\CalH(L)$. We write $\CalC_R(D)$ and
$\CalH_R(L)$ if we want to emphasize the ring of coefficients that we work
with. If $R$ is omitted from the notation, integer coefficients are assumed.
\end{defin}

\subsection{Turner spectral sequence}
In the case of $R=\Z_2$, Paul Turner~\cite{Turner-diff} defined another
differential $d_T:\CalC_{\Z_2}(D)\to\CalC_{\Z_2}(D)$ on the Khovanov chain
complex over $\Z_2$. Its definition follows the one above for $d$ almost
verbatim except that the multiplication and comultiplication maps are
different:
\begin{equation}\label{eq:T-diff}
\begin{gathered}
m_T(1\otimes 1)=m_T(1\otimes X)=m_T(X\otimes1)=0,\qquad m_T(X\otimes X)=X;\\
\GD_T(1)=1\otimes 1,\qquad \GD_T(X)=0.
\end{gathered}
\end{equation}

Main properties of $d_T$ are summarized in the theorem below.

\begin{thm}[Turner~\cite{Turner-diff}]\label{thm:Turner-spectral}
Let $L$ be a link with $c$ components represented by a diagram $D$. Then
\begin{enumerate}
\item $d_T$ is a differential on $\CalC_{\Z_2}(D)$ of bidegree $(1,2)$;
\item $d_T$ commutes with $d$ and induces a well-defined differential $d_T^*$
of bidegree $(1,2)$ on the Khovanov homology groups $\CalH_{\Z_2}(L)$;
\item $(\CalC_{\Z_2}(D),d,d_T)$ has a structure of a double complex and,
hence, results in a spectral sequence $\{(E^T_r,d_r)\}_{r\ge0}$ with
$(E^T_0,d_0)=(\CalC_{\Z_2}(D),d)$ and $(E^T_1,d_1)=(\CalH_{\Z_2}(L),d_T^*)$,
where the bidegree of $d_r$ equals $(1,2r)$;
\item $\{(E^T_r,d_r)\}$ converges to $E^T_\infty$ that has the total dimension
over $\Z_2$ equal $2^c$;
\item if $L$ is $\Z_2$H-thin, then $\{(E^T_r,d_r)\}$ collapses at
the second page, that is, $d_r=0$ for $r\ge2$ and $E^T_2=E^T_\infty$.
\end{enumerate}
\end{thm}

\begin{defin}
The spectral sequence from Theorem~\ref{thm:Turner-spectral} is called the
{\em Turner spectral sequence} for a link $L$.
\end{defin}

\subsection{Generators of $\CalC(D)$ and another differential on
$\CalH_{\Z_2}(L)$}\label{sec:generators}
Let $L$ be a link represented by a diagram $D$ and let $s$ be its Kauffman
state. Then $\CalA(s)$ is freely generated as an $R$-module by $2^{|D_s|}$
generators of the form $a_1\otimes a_2\otimes\cdots\otimes a_{|D_s|}$, where
each $a_i$ is either $1$ or $X$.

Define an $R$-linear homomorphism $\Gn_s:\CalA(s)\to\CalA(s)$ of degree $2$ by
specifying that on a generator of $\CalA(s)$ it equals the sum of all the
possibilities to replace an $X$ in this generator with $1$. For example, if
$|D_s|=3$, then $\Gn_s(X\otimes X\otimes 1)=
1\otimes X\otimes 1+X\otimes 1\otimes 1$,
$\Gn_s(1\otimes X\otimes 1)=1\otimes1\otimes1$ and 
$\Gn_s(1\otimes 1\otimes 1)=0$.
We extend this to a map $\Gn:\CalC(D)\to\CalC(D)$ of bidegree $(0,2)$ as
$\Gn=\sum_s\Gn_s$.

Restrict now our attention to the case of $R=\Z_2$. Main properties of
$\Gn$ over $\Z_2$ are listed below. They are proved in~\cite{me-torsion}.

\begin{thm}[\cite{me-torsion}]\label{thm:me-diff}
Let $L$ be a link and $D$ its diagram. Then
\begin{enumerate}
\item $\Gn$ is a differential on $\CalC_{\Z_2}(D)$ of bidegree $(0,2)$;
\item $\Gn$ commutes with $d$ and, hence, induces a differential $\Gn^*$
of bidegree $(0,2)$ on $\CalH_{\Z_2}(L)$;
\item both $\Gn$ and $\Gn^*$ are acyclic;
\item if $L$ is $\Z_2$H-thin, then $\Gn^*$ establishes an isomorphism between
the two non-trivial diagonals of $\CalH_{\Z_2}(L)$.
\end{enumerate}
\end{thm}

\begin{rem}
$\Gn^*$ helps to establish the fact that $\CalH_{\Z_2}(L)$ is isomorphic to a
direct sum of two copies of the reduced Khovanov homology of $L$ over
$\Z_2$, see~\cite{me-torsion}.
\end{rem}

\def\stmodp{\mkern -12mu\mod p}
\section{Relations between differentials on Khovanov homology over $\Z_2$}
\label{sec:Bockstein}
\subsection{Bockstein spectral sequence}
We start this sections by recalling main definitions and properties of the
Bockstein spectral sequence following Chapter~10
of~\cite{McCleary-Bockstein}. Our setup differs slightly from the one
in~\cite{McCleary-Bockstein} since we consider abstract chain complexes,
while \cite{McCleary-Bockstein} deals with singular homology of topological
spaces. Nonetheless, all the relevant statements still hold true as long as we
work with finitely generated complexes of free Abelian groups. We state all
the results for an arbitrary prime $p$, but will only need them for $p=2$.

Let
$\CalC=\Bigl(\cdots\lra\CalC^{i-1}\lra\CalC^i\lra\CalC^{i+1}\lra\cdots\Bigr)$
be a (co)chain complex of free Abelian groups and let $p$ be a prime number.
A short exact sequence \hbox{$0\lra\Z_p\stackrel{\times p}{\vlra}\Z_{p^2}
\stackrel{\stmodp}{\vlra}\Z_p\lra0$} of coefficient rings induces a
short exact sequence \hbox{$0\lra\CalC\otimes\Z_p\stackrel{\times p}{\vlra}
\CalC\otimes\Z_{p^2}\stackrel{\stmodp}{\vlra}\CalC\otimes\Z_p\lra0$}
of complexes that, in turn, results in a long exact sequence of the
corresponding (co)homology groups. Let $\Gb:H(\CalC;\Z_p)\to H(\CalC;\Z_p)$ be
a connecting homomorphism in this long exact sequence. It has homological
degree $1$. We call $\Gb$ the {\em Bockstein homomorphism} on $H(\CalC;\Z_p)$.

\begin{thm}[cf. Theorem~10.3 and Proposition~10.4 of \cite{McCleary-Bockstein}]
\label{thm:Bockstein}
Let $\CalC$ be a finitely generated (co)chain complex of free Abelian groups.
Then there exists a singly-graded spectral sequence 
$\{(B_r,b_r)\}_{r\ge1}$ with $(B_1,b_1)=(H(\CalC;\Z_p),\Gb)$ that converges to
$(H(\CalC)/\hbox{\rm torsion})\otimes\Z_p$. Moreover, $B_r\simeq\Im\bigl(
H(\CalC;\Z_{p^r})\stackrel{\times p^{r-1}}{\vlra}H(\CalC;\Z_{p^r})\bigr)$, a
graded subgroup of $H(\CalC;\Z_{p^r})$, and $b_r$ can be identified with the
connecting homomorphism induced on $H(\CalC;\Z_{p^r})$ from a short exact
sequence of coefficient rings
\hbox{$0\to\Z_{p^r}\lra\Z_{p^{2r}}\lra\Z_{p^r}\to0$}.
Also, $\deg(b_r)=1$ for every $r\ge1$.
\end{thm}

\begin{defin}
$\{(B_r,b_r)\}$ is called the {\em Bockstein spectral sequence} of
$H(\CalC;\Z_p)$ and $b_r$ is called the {\em $r$-th order Bockstein
homomorphism}.
\end{defin}

\begin{cor}\label{cor:Bock-collapse}
The Bockstein spectral sequence $\{(B_r,b_r)\}$ collapses at the $r$-th page
if and only if $H(\CalC)$ has no torsion elements of order $2^k$ for $k\ge r$.
\end{cor}
\begin{proof}
$\{(B_r,b_r)\}$ collapses at the $r$-th page if and only if
$\dim_{\Z_p}B_r=\dim_{\Z_p}B_\infty$. By Theorem~\ref{thm:Bockstein},
$\dim_{\Z_p}B_r$ equals the number of $\Z_{p^r}$ factors in $H(\CalC;\Z_{p^r})$
and $\dim_{\Z_p}B_\infty$ equals the number of $\Z$ factors in $H(\CalC)$.
The rest follows from the Universal Coefficient Theorem.
\end{proof}

\begin{attn}\label{attn:beta-formula}
The Bockstein differential $\Gb$ has another description that is more
useful for our purposes. Namely, if $b:H(\CalC;\Z_p)\to H(\CalC)$ is
the connecting homomorphisms induced from a short exact sequence of
coefficients
$0\to\Z\stackrel{\times p}{\vlra}\Z\stackrel{\stmodp}{\vlra}\Z_p\to0$,
then $\Gb=b\mod p$.
\end{attn}

\subsection{Proof of the main result} 
Let $p=2$ from now on.
\begin{lem}\label{lem:diff-commucator}
Let $L$ be a link. Then $d^*_T=\Gb\circ\Gn^*+\Gn^*\circ\Gb$ on
$\CalH_{\Z_2}(L)$.
\end{lem}
We postpone the proof of this lemma until the next section.

\begin{proof}[Proof of Theorem~\ref{thm:main-result}]
Recall that our task is to prove that if $L$ is $\Z_2$H-thin, then $\CalH(L)$
contains no torsion elements of order $2^k$ for $k>1$. By
Corollary~\ref{cor:Bock-collapse}, it is enough to show that the
Bockstein spectral sequence collapses at the second page, that is,
$\dim_{\Z_2}B_2=\dim_{\Z_2}B_\infty=\rk\CalH_\Q(L)$. Now,
$\rk\CalH_\Q(L)=|\tJ_L(\sqrt{-1})|+2^{c-1}$, where $c$ is the number of
components of $L$, by Theorem~\ref{thm:lee-pattern} and
$\dim_{\Z_2}B_1=\dim_{\Z_2}(\CalH_{\Z_2}(L))=2|\tJ_L(\sqrt{-1})|$
by Theorem~\ref{thm:me-pattern}. Since $b_1=\Gb$, it only remains to show that
$\rk_{\Z_2}\Gb=\frac12(\dim_{\Z_2}B_1-\dim_{\Z_2}B_2)=
\frac12|\tJ_L(\sqrt{-1})|-2^{c-2}$.

Since $L$ is $\Z_2$H-thin, the Turner spectral sequence collapses at the
second page, that is, $\dim_{\Z_2}E^T_2=\dim_{\Z_2}E^T_\infty=2^c$ by 
Theorem~\ref{thm:Turner-spectral}. Since $E^T_1=\CalH_{\Z_2}(L)$ and
$d_1=d^*_T$, we have that $\rk_{\Z_2}(d^*_T)=
\frac12\bigl(\dim_{\Z_2}E^T_1-\dim_{\Z_2}E^T_2\bigr)=
|\tJ_L(\sqrt{-1})|-2^{c-1}$.

Let $\CalH_l$ and $\CalH_u$ be (graded) subgroups of $\CalH_{\Z_2}(L)$
supported on the lower and upper diagonals, respectively. Then
$\CalH_{\Z_2}(L)\simeq\CalH_l\oplus\CalH_u$ since $L$ is $\Z_2$H-thin.
Finally, we observe that $\Gn^*\big|_{\CalH_l}$ is an isomorphism between
$\CalH_l$ and $\CalH_u$ by Theorem~\ref{thm:me-diff},
$\Gn^*\big|_{\CalH_u}=0$, $\Gb\big|_{\CalH_l}=0$, and
$\rk_{\Z_2}\Gb=\rk_{\Z_2}\bigl(\Gb\big|_{\CalH_u}\bigr)$
since $\Gb$ is trivial everywhere else. We conclude that
$\rk_{\Z_2}\bigl(d^*_T\big|_{\CalH_l}\bigr)=
\rk_{\Z_2}\bigl(\Gb\circ\Gn^*\big|_{\CalH_l}\bigr)=
\rk_{\Z_2}\bigl(\Gb\big|_{\CalH_u}\bigr)=\rk_{\Z_2}\Gb$ and
$\rk_{\Z_2}\bigl(d^*_T\big|_{\CalH_u}\bigr)=
\rk_{\Z_2}\bigl(\Gn^*\circ\Gb\big|_{\CalH_u}\bigr)=
\rk_{\Z_2}\bigl(\Gb\big|_{\CalH_u}\bigr)=\rk_{\Z_2}\Gb$.
Hence, $\rk_{\Z_2}(d^*_T)=\rk_{\Z_2}\bigl(d^*_T\big|_{\CalH_l}\bigr)+
\rk_{\Z_2}\bigl(d^*_T\big|_{\CalH_u}\bigr)=2\rk_{\Z_2}\Gb$.
\end{proof}

Proof of Theorem~\ref{thm:main-result} suggests that there is a deeper 
relation between Turner and Bockstein spectral sequences on $\CalH_{\Z_2}(L)$.
We end this section with a couple of conjectures.

\begin{myconj}
There exists an algebraic relation between Turner and Bockstein differentials
on higher pages in the corresponding spectral sequences. This relation should
involve higher order generalizations of $\Gn$.
\end{myconj}

\begin{myconj}
If the Turner spectral sequence on $\CalH_{\Z_2}(L)$ collapses at the $r$-th
page, then the Bockstein one collapses at the $r$-th page as well. In
particular, if $\CalH_{\Z_2}(L)$ is supported on $r$ adjacent diagonals, then
$\CalH(L)$ does not have torsion elements of order $2^k$ for $k\ge r$.
\end{myconj}

\section{Proof of Lemma~\ref{lem:diff-commucator}}
\label{sec:lemma-proof}

Recall that we have to prove that $d^*_T=\Gb\circ\Gn^*+\Gn^*\circ\Gb$ on
$\CalH_{\Z_2}(L)$. For a (co)chain $\Bc\in\CalC_{\Z_2}(D)$ and a generator
$\Bx$ of $\CalC(D)$ (see Section~\ref{sec:generators}), denote by $\Bc|_\Bx$
the coefficient of $\Bx$ in $\Bc$. We say that {\em $\Bx$ is from $\Bc$} and
write $\Bx\in\Bc$ if $\Bc|_\Bx=1$ and we say that {\em $\Bx$ is not from
$\Bc$} and write $\Bx\not\in\Bc$ otherwise. Finally, we denote the Kauffman
state that corresponds to $\Bx$ by $s(\Bx)$.

Fix a bigrading $(i,j)$ on $\CalC(D)$ so that
$d_T:\CalC^{i,j}_{\Z_2}(D)\to\CalC^{i+1,j+2}_{\Z_2}(D)$.
For a (co)homology class $[\Bc]_2\in\CalH^{i,j}_{\Z_2}(L)$ represented by a
(co)cycle $\Bc\in\CalC^{i,j}_{\Z_2}$, the connecting homomorphism $b$
from~\ref{attn:beta-formula} is defined as
$b([\Bc]_2)=\bigl[\frac12d(\Bc)\bigr]\in\CalH^{i+1,j}(L)$, where $\Bc$ is
lifted to $\CalC(D)$ so that $d(\Bc)$ makes sense. It follows that
\begin{equation}\label{eq:beta-def}
\Gb([\Bc]_2)=\left[\frac12d(\Bc)\right]_2\in\CalH^{i+1,j}_{\Z_2}(L),
\end{equation}
where $[\,\bullet\,]_2$ means taking the homology class over $\Z_2$.
Let $\Gd:\CalC^{i,j}(D)\to\CalC^{i+1,j+2}(D)$ be the homomorphism defined as
$\Gd=d\circ\Gn+\Gn\circ d$. Formula \eqref{eq:beta-def} implies that
\begin{equation}\label{eq:bn+nb}
(\Gb\circ\Gn^*+\Gn^*\circ\Gb)([\Bc]_2)=\left[\frac12\Gd(\Bc)\right]_2=
\left[\frac12\sum_{\Bx\in\Bc}\Gd(\Bx)\right]_2.
\end{equation}

Given a generator $\Bx$ of $\CalC^{i,j}(D)$, and a generator $\Bz$ of 
$\CalC^{i+1,j+2}(D)$, denote by $\CalL_\Bx^\Bz$ the set of all generators
$\By$ of $\CalC^{i+1,j}(D)$ such that $\By\in d(\Bx)$ and $\Bz\in\Gn(\By)$.
Similarly, denote by $\CalU_\Bx^\Bz$ the set of all generators
$\By'$ of $\CalC^{i,j+2}(D)$ such that $\By'\in\Gn(\Bx)$ and $\Bz\in d(\By)$.
We denote elements of $\CalL_\Bx^\Bz$ and $\CalU_\Bx^\Bz$ by
\hbox{$\Bx\to\By\to\Bz$} and \hbox{$\Bx\to\By'\to\Bz$}, respectively.
Since $\Gn$ does not change the underlying Kauffman states of the generators, 
$\CalL_\Bx^\Bz$ and $\CalU_\Bx^\Bz$ are empty when states $s(\Bx)$ and
$s(\Bz)$ are not adjacent. Hence, we are going to assume that $s(\Bx)$ and
$s(\Bz)$ are adjacent states from now on.
By definition,
$\Gd(\Bx)|_\Bz=\Ge_\Bx^\Bz\bigl(|\CalL_\Bx^\Bz|+|\CalU_\Bx^\Bz|\bigr)$,
where $|\bullet|$ denotes the cardinality of a set and
$\Ge_\Bx^\Bz=\Ge(s(\Bx),s(\Bz))$, see Definition~\ref{def:Khovanov}.

Let $\widetilde\CalL_\Bx^\Bz\subset\CalL_\Bx^\Bz$ consist of all
$(\Bx\to\By\to\Bz)\in\CalL_\Bx^\Bz$ such that the circle of $s(\By)$ where
$X$ is replaced with $1$ in $\Bz$ does not pass through the crossing at which
$s(\Bx)$ and $s(\By)$ differ.
We define $\widetilde\CalU_\Bx^\Bz\subset\CalU_\Bx^\Bz$ similarly.
Let $\widehat\CalL_\Bx^\Bz=\CalL_\Bx^\Bz\setminus\widetilde\CalL_\Bx^\Bz$ and
$\widehat\CalU_\Bx^\Bz=\CalU_\Bx^\Bz\setminus\widetilde\CalU_\Bx^\Bz$.
There is a natural bijection between sets $\widetilde\CalL_\Bx^\Bz$ and
$\widetilde\CalU_\Bx^\Bz$ since changes made under $d$ and $\Gn$ take place
on circles that do not interfere with each other. Hence,
$|\widetilde\CalL_\Bx^\Bz|=|\widetilde\CalU_\Bx^\Bz|$.

We list elements of $\widehat\CalL_\Bx^\Bz$ and $\widehat\CalU_\Bx^\Bz$ 
for all possible $\Bx$ and $\Bz$ in Figure~\ref{fig:LU-xz}, where we omit
common parts of $\Bx$, $\Bz$, $\By$ and $\By'$.
We classify pairs $(\Bx,\Bz)$ as being of four different types, $A$, $B$,
$C_m$ and $C_\GD$, depending on the outcome (see Figure~\ref{fig:LU-xz}) and
denote the type of $(\Bx,\Bz)$ by $t(\Bx,\Bz)$.
Then $|\widehat\CalL_\Bx^\Bz|=|\widehat\CalU_\Bx^\Bz|=0$ if $t(\Bx,\Bz)=A$
and $|\widehat\CalL_\Bx^\Bz|=|\widehat\CalU_\Bx^\Bz|=1$ if $t(\Bx,\Bz)=B$.
If $t(\Bx,\Bz)=C_m$, then $|\widehat\CalL_\Bx^\Bz|=0$ and
$|\widehat\CalU_\Bx^\Bz|=2$, and
if $t(\Bx,\Bz)=C_\GD$, then $|\widehat\CalL_\Bx^\Bz|=2$ and
$|\widehat\CalU_\Bx^\Bz|=0$. It follows that
\begin{equation}\label{eq:UL-types}
\begin{split}
\sum_{\Bx\in\Bc}\Ge_\Bx^\Bz|\widehat\CalL_\Bx^\Bz|&=
\sum_{\substack{\Bx\in\Bc\\t(\Bx,\Bz)=B}}\!\!\!\!\Ge_\Bx^\Bz+
  \sum_{\substack{\Bx\in\Bc\\t(\Bx,\Bz)=C_\GD}}\!\!\!\!2\Ge_\Bx^\Bz;\\[2mm]
\sum_{\Bx\in\Bc}\Ge_\Bx^\Bz|\widehat\CalU_\Bx^\Bz|&=
\sum_{\substack{\Bx\in\Bc\\t(\Bx,\Bz)=B}}\!\!\!\!\Ge_\Bx^\Bz+
  \sum_{\substack{\Bx\in\Bc\\t(\Bx,\Bz)=C_m}}\!\!\!\!2\Ge_\Bx^\Bz.
\end{split}
\end{equation}

\begin{figure}[t]
\def\arraystretch{1.2}%
$\begin{array}{|c||c|c|c|c|}
\stdhline
\vrule width 0pt height 12pt depth 5pt
\hbox{type}&\Bx&\Bz&\widehat\CalL_\Bx^\Bz&\widehat\CalU_\Bx^\Bz\\
\stdhline
A&1\otimes 1&1&\varnothing&\varnothing\\
A&1\otimes 1&X&\varnothing&\varnothing\\
B&1\otimes X&1&1\otimes X\to X\to 1&1\otimes X\to 1\otimes1\to 1\\
A&1\otimes X&X&\varnothing&\varnothing\\
B&X\otimes 1&1&X\otimes 1\to X\to 1&X\otimes 1\to 1\otimes1\to 1\\
A&X\otimes 1&X&\varnothing&\varnothing\\
A&X\otimes X&1&\varnothing&\varnothing\\
C_m&X\otimes X&X&\varnothing&X\otimes X\to 1\otimes X\to X\\
   &          & &&X\otimes X\to X\otimes 1\to X\\
\stdhline
C_\GD&1&1\otimes 1&1\to 1\otimes X\to 1\otimes 1&\varnothing\\
     & &          &1\to X\otimes 1\to 1\otimes 1&\\
A&1&1\otimes X&\varnothing&\varnothing\\
A&1&X\otimes 1&\varnothing&\varnothing\\
A&1&X\otimes X&\varnothing&\varnothing\\
A&X&1\otimes 1&\varnothing&\varnothing\\
B&X&1\otimes X&X\to X\otimes X\to 1\otimes X&X\to 1\to 1\otimes X\\
B&X&X\otimes 1&X\to X\otimes X\to X\otimes 1&X\to 1\to X\otimes 1\\
A&X&X\otimes X&\varnothing&\varnothing\\
\stdhline
\end{array}$
\caption{$\widehat\CalL_\Bx^\Bz$ and $\widehat\CalU_\Bx^\Bz$ for different
types of $(\Bx$, $\Bz)$}
\label{fig:LU-xz}
\end{figure}

Therefore, for every $[\Bc]_2\in\CalH^{i,j}_{\Z_2}(L)$ and every generator
$\Bz$ of $\CalC^{i+1,j+2}(D)$ we have:
\begin{equation}\label{eq:delta-comp}
\begin{split}
\frac12\Gd(\Bc)\big|_\Bz&=
\frac12\sum_{\Bx\in\Bc}\Gd(\Bx)\big|_\Bz\\
&=\frac12\sum_{\Bx\in\Bc}\Ge_\Bx^\Bz
  \bigl(|\CalL_\Bx^\Bz|+|\CalU_\Bx^\Bz|\bigr)\\
&=\frac12\sum_{\Bx\in\Bc}\Ge_\Bx^\Bz
  \bigl(|\widetilde\CalL_\Bx^\Bz|+|\widetilde\CalU_\Bx^\Bz|\bigr)
 +\frac12\sum_{\Bx\in\Bc}\Ge_\Bx^\Bz
  \bigl(|\widehat\CalL_\Bx^\Bz|+|\widehat\CalU_\Bx^\Bz|\bigr)\\
&=\sum_{\Bx\in\Bc}\Ge_\Bx^\Bz|\widetilde\CalL_\Bx^\Bz|
 +\sum_{\substack{\Bx\in\Bc\\t(\Bx,\Bz)=B}}\!\!\!\!\Ge_\Bx^\Bz
 +\sum_{\substack{\Bx\in\Bc\\t(\Bx,\Bz)=C_\GD}}\!\!\!\!\Ge_\Bx^\Bz
 +\sum_{\substack{\Bx\in\Bc\\t(\Bx,\Bz)=C_m}}\!\!\!\!\Ge_\Bx^\Bz,
\end{split}
\end{equation}
where the last equality follows from \eqref{eq:UL-types} and the fact that
$|\widetilde\CalL_\Bx^\Bz|=|\widetilde\CalU_\Bx^\Bz|$.

Since $\Bc$ is a (co)cycle modulo $2$, for every generator $\By$ of
$\CalC^{i+1,j}(D)$ we have that $d(\Bc)|_\By$ is an even number. It follows
that $\sum_{\Bx\in\Bc}\Ge_\Bx^\Bz|\CalL_\Bx^\Bz|$ is even as well.
From~\eqref{eq:UL-types} we have:
\begin{equation}
\sum_{\Bx\in\Bc}\Ge_\Bx^\Bz|\CalL_\Bx^\Bz|=
\sum_{\Bx\in\Bc}\Ge_\Bx^\Bz|\widetilde\CalL_\Bx^\Bz|+
  \sum_{\Bx\in\Bc}\Ge_\Bx^\Bz|\widehat\CalL_\Bx^\Bz|=
\sum_{\Bx\in\Bc}\Ge_\Bx^\Bz|\widetilde\CalL_\Bx^\Bz|+
  \!\!\!\!\sum_{\substack{\Bx\in\Bc\\t(\Bx,\Bz)=B}}\!\!\!\!\Ge_\Bx^\Bz+
  \;2\!\!\!\!\!\sum_{\substack{\Bx\in\Bc\\t(\Bx,\Bz)=C_\GD}}\!\!\!\!\Ge_\Bx^\Bz.
\end{equation}
Therefore, $\displaystyle\sum_{\Bx\in\Bc}\Ge_\Bx^\Bz|\widetilde\CalL_\Bx^\Bz|+
\!\!\!\!\sum_{\substack{\Bx\in\Bc\\t(\Bx,\Bz)=B}}\!\!\!\!\Ge_\Bx^\Bz$ is an
even number. Denote it by $2N$ for some $N\in\Z$.
It follows from~\eqref{eq:delta-comp} that
\begin{equation}\label{eq:delta-comp2}
\begin{split}
\frac12\Gd(\Bc)\big|_\Bz&=
2N+\!\!\!\!\sum_{\substack{\Bx\in\Bc\\t(\Bx,\Bz)=C_\GD}}\!\!\!\!\Ge_\Bx^\Bz
 +\!\!\!\!\sum_{\substack{\Bx\in\Bc\\t(\Bx,\Bz)=C_m}}\!\!\!\!\Ge_\Bx^\Bz\\
&\equiv \sum_{\substack{\Bx\in\Bc\\t(\Bx,\Bz)=C_\GD}}\!\!\!\!1
 +\!\!\!\!\sum_{\substack{\Bx\in\Bc\\t(\Bx,\Bz)=C_m}}\!\!\!\!1\mod2\\[2mm]
&\equiv d_T(\Bc)\big|_\Bz\mod2
\end{split}
\end{equation}
by the definition~\eqref{eq:T-diff} of $d_T$.

Since this is true for every generator $\Bz$ of $\CalC^{i+1,j+2}(D)$, we have
that
\begin{equation}
d^*_T([\Bc]_2)=[d_T(\Bc)]_2=\left[\frac12\Gd(\Bc)\right]_2=
(\Gb\circ\Gn^*+\Gn^*\circ\Gb)([\Bc]_2)
\end{equation}
for every $[\Bc]_2\in\CalH^{i,j}_{\Z_2}(L)$ and, hence,
$d^*_T=\Gb\circ\Gn^*+\Gn^*\circ\Gb$, as desired.\qed

\raggedright


\begin{thebibliography}{mmm}
\bibitem[AP]{Asaeda-Przytycki}
M.~Asaeda and J.~Przytycki, {\sl Khovanov homology: torsion and thickness} in
    Advances in topological quantum field theory, 135--166, Kluwer Acad.
    Publ., Dordrecht, 2004; arXiv:math.GT/0402402.

\bibitem[BN]{BN-first}
D.~Bar-Natan, {\sl On Khovanov's categorification of the Jones polynomial},
   Alg. Geom. Top., {\bf 2} (2002) 337--370; arXiv:math.QA/0201043.

\bibitem[G]{Garoufalidis}
S.~Garoufalidis, {\sl A conjecture on Khovanov's invariants}, Fund. Math.
    {\bf 184} (2004), 99--101.

\bibitem[HTh]{Knotscape}
J.~Hoste and M.~Thistlethwaite, {\tt Knotscape} --- a program for
    studying knot theory and providing convenient access to tables of
    knots, {\tt http://www.math.utk.edu/\~{}morwen/knotscape.html}

\bibitem[J]{Jones}
V.~Jones, {\sl A polynomial invariant for knots via von Neumann algebras},
   Bull. Amer. Math. Soc. {\bf 12} (1985), 103--111.

\bibitem[Kh1]{Kh-Jones}
M.~Khovanov, {\sl A categorification of the Jones polynomial},
    Duke Math. J. {\bf 101} (2000), no.~3, 359--426; arXiv:math.QA/9908171.

\bibitem[Kh2]{Kh-patterns}
M.~Khovanov, {\sl Patterns in knot cohomology I}, Experiment. Math. {\bf 12}
    (2003), no. 3, 365--374; arXiv:math.QA/0201306.

\bibitem[Kh3]{Kh-Frobenius}
M.~Khovanov, {\sl Link homology and Frobenius extensions},
   Fundamenta Mathematicae, {\bf 190} (2006), 179--190; arXiv:math.QA/0411447.

\bibitem[L1]{Lee-H_slim}
E.~S.~Lee, {\sl The support of the Khovanov's invariants for alternating
    knots}, arXiv:math.GT/0201105.

\bibitem[L2]{Lee-patterns}
E.~S.~Lee, {\sl An endomorphism of the Khovanov invariant}, Adv. Math.
    {\bf 197} (2005), no. 2, 554--586; arXiv:math.GT/0210213.

\bibitem[MC]{McCleary-Bockstein}
J.~McCleary, {\sl A user's guide to spectral sequences}, Cambridge Studies in
    Adv. Math., {\bf 58}, Cambridge Univ. Press, Cambridge, 2001.

\bibitem[PPS]{Jozef-Radmila+Milena-torsion}
M.~Pabiniak, J.~Przytycki, and R.~Sazdanovi\'c, {\sl On the first group of the
    chromatic cohomology of graphs}, Geom. Dedicata, {\bf 140} (2009), no. 1,
    19--48; arXiv:math.GT/0607326.

\bibitem[PS]{Jozef-Radmila-torsion}
J.~Przytycki and R.~Sazdanovi\'c, {\sl Torsion in Khovanov homology of
    semi-adequate links}, to appear in Fund. Math.; arXiv:1210.5254.

\bibitem[Sh]{me-torsion}
A.~Shumakovitch, {\sl Torsion of Khovanov homology},
    Fund. Math. {\bf 225} (2014), 343--364; arXiv:math.GT/0405474.

\bibitem[Tu]{Turner-diff}
P.~Turner, {\sl Calculating Bar-Natan's characteristic two Khovanov homology},
   J. Knot Th. Ramif. {\bf 15} (2006), no. 10, 1335--1356; arXiv:math/0411225.

\end{thebibliography}
\end{document}